\documentclass[letterpaper, conference,9pt]{IEEEtran}      

\usepackage{mathptmx} 
\usepackage{times} 
\usepackage{amsmath} 
\usepackage{amsbsy} 
\usepackage{amssymb}
\usepackage{mathrsfs}
\usepackage{comment}
\usepackage[export]{adjustbox}
\usepackage{tikz}
\usetikzlibrary{external,positioning,decorations.pathreplacing,shapes,arrows,patterns}

\usepackage{algorithmicx}
\usepackage{pgfplots}
\usepackage{graphicx}
\usepackage{pstool}
\usepackage[latin1]{inputenc}
\usetikzlibrary{arrows,shapes}
\usepackage{xifthen}
\usepackage{epic}
\usepackage{caption}
\usepackage{epstopdf}

\newtheorem{thm}{\bf{Theorem}}

\newtheorem{lem}[thm]{\bf {Lemma}}
\newtheorem{prop}[thm]{\bf {Proposition}}

\newenvironment{proof}{\paragraph*{Proof}}{\hfill$\square$ \newline}
\newcommand{\sgn}{\mathrm{sgn} }

\newcommand*{\QEDA}{\hfill\ensuremath{\square}}

\tikzstyle{int}=[draw, fill=blue!10, minimum height = 1cm, minimum width=1.5cm,thick ]
\tikzstyle{sint}=[draw, fill=blue!10, minimum height = 0.5cm, minimum width=0.8cm,thick ]
\tikzstyle{sum}=[circle, fill=blue!10, draw=black,line width=1pt,minimum size = 0.5cm, thick ]
\tikzstyle{ssum}=[circle, fill=blue!10,draw=black,line width=1pt,minimum size = 0.1cm]
\tikzstyle{int1}=[draw, fill=blue!10, minimum height = 0.5cm, minimum width=1cm,thick ]
\tikzstyle{enc}=[draw, fill=blue!10, minimum height = 2.7cm, minimum width=1cm,thick ]
\tikzstyle{int}=[draw, fill=blue!10, minimum height = 1cm, minimum width=1.5cm,thick ]

\title{\LARGE \bf Mean Estimation from Adaptive One-bit Measurements}
\author{
\IEEEauthorblockN{Alon Kipnis}
\IEEEauthorblockA{Department of Statistics \\
Stanford University\\
Stanford, CA\\}
\and
\IEEEauthorblockN{John C. Duchi}
\IEEEauthorblockA{Department of Electrical Engineering \\
and Department of Statistics \\
Stanford University\\
Stanford, CA\\}
}

\begin{document}
\graphicspath{{../Figs/}}
\maketitle
\thispagestyle{empty}
\pagestyle{empty}

\begin{abstract}
We consider the problem of estimating the mean of a normal distribution under the following constraint: the estimator can access only a single bit from each sample from this distribution. We study the squared error risk in this estimation as a function of the number of samples and one-bit measurements $n$. We consider an adaptive estimation setting where the single-bit sent at step $n$ is a function of both the new sample and the previous $n-1$ acquired bits. For this setting, we show that no estimator can attain asymptotic mean squared error smaller than $\pi/(2n)+O(n^{-2})$ times the variance. 
In other words, one-bit restriction increases the number of samples required for a prescribed accuracy of estimation by a factor of at least $\pi/2$ compared to the unrestricted case.
%
In addition, we provide an explicit estimator that attains this asymptotic error, showing that, rather surprisingly, only $\pi/2$ times more samples are required in order to attain estimation performance equivalent to the unrestricted case. 
\end{abstract}



\section{Introduction}
\label{sec:Intro}

The performance in estimating information from data collected and processed by multiple units may be limited due to communication constraints between these units. 
For example, consider large-scale sensor arrays where information is collected at multiple physical locations and transmitted to a central estimation unit. In this scenario, the ability to estimate a particular parameter from the data is dictated not only by the quality of observations and their number, but also by the available rate for communication between the sensors and the central estimator. The question that we ask is to what extent a parametric estimation task is affected by these communication constraints, and what are the fundamental performance limits in estimating a parameter subject to these restrictions. In this paper we answer this question in a particular setting: the estimation of the mean $\theta$ of a normal distribution with variance $\sigma^2$ under the constraint that only a single bit can be communicated on each sample $X_n$ from this distribution. As it turns out, the ability to share information among different samples before committing on each single-bit message dramatically affects the performance in estimating $\theta$. We therefore distinguish among three settings:
 \begin{itemize}
 \item[(i)] \emph{Centralized} encoding: all $n$ encoders confer and produce a single $n$ bit message which is a function of $X_1,\ldots,X_n$. 
 \item[(ii)] \emph{Adaptive} or \emph{sequential} encoding: the $n$th encoder observes $X_n$ and the $n-1$ previous single bit messages.
 \item[(iii)] \emph{Distributed} encoding: the output of the $n$th encoder is a single bit that is only a function of $X_n$.
 \end{itemize}
Clearly, as far as information sharing is concerned, settings (iii) is a more restrictive version of (ii) which is more restrictive than (i). We measure the estimation performance by the mean squared error (MSE) risk. We are interested in particular in the \emph{asymptotic relative efficiency} (ARE) of estimators in the constrained setting relative to the MSE attained by the empirical mean of the samples, which is the minimax estimator in estimating without one-bit constraint and its MSE decreases as $\sigma^2/n+O(n^{-2})$. \\

In setting (i), the estimator can evaluate the empirical mean of the samples and then communicate it using $n$ bits. This strategy leads to MSE behavior of $\sigma^2/n + O(2^{-n})$. Therefore, the ARE in this setting is $1$. Namely, asymptotically, there is no loss in performance due to the communication constraint under centralized encoding. In this work we show that a similar result does not hold even in setting (ii): the ARE of any adaptive estimation scheme is at least $\pi/2$. Namely, the single-bit per sample constraint incurs a minimal penalty of at least $1.57$ in the number of samples compared to an unconstrained estimator or to the optimal estimator in setting (i). In addition to this negative statement, we provide an estimator that attains this ARE. In other words, we show that the lower bound of $\pi/2$ on the ARE is tight, and that it is attained regardless of the particular realization of $\theta$ or the radius of the parameter space from which it is taken. Clearly, the minimal penalty on the efficiency of $\pi/2$ also holds under setting (iii), although the question whether this efficiency is achievable (or otherwise, what is the minimal ARE) remains open. 
\\

The lower bound of $\pi/2$ on the ARE, i.e., a lower bound of $\sigma^2\pi/(2n) +O(n^{-2})$ on the MSE, is obtained by showing that the Fisher information of any $n$ adaptive messages is not greater than $2n/(\pi \sigma^2)$. From here, the desired bound on the MSE follows from the van Trees version of the information inequality \cite{gill1995applications}. Finally, we show that an estimator that attains asymptotic MSE of $\sigma^2\pi/(2n)$ is obtained as a special case of \cite[Thm. 4]{polyak1992acceleration}. In addition to these two results, we also derive the one-step optimal strategy in which the message sent at step $i$th is designed to minimize the MSE given this message and the previous $i-1$ messages. Furthermore, we demonstrates numerically that the MSE under this strategy converge to $\pi \sigma^2/(2n)$.\\

We note that even though its ARE is $1$, the centralized encoding setting (i) already poses a non-trivial challenge for the design and analysis of an optimal encoding and estimation scheme. Indeed, the standard technique to encode an unknown random quantity using $n$ bits is equivalent to the design of a scalar quantizer \cite{gray1998quantization}. However, the optimal design of this quantizer depends on the distribution of its input, which is the goal of our estimation problem and hence its exact value is unknown. As a result, a non-trivial exploration exploitation tradeoff arises in this case. Note that the only missing parameter in our setting is the mean, which, under setting (i), is known to the encoder with uncertainty interval proportional to $\sigma/\sqrt{n}$. Therefore, while it is clear that uncertainty due to quantization decreases exponentially in the number of bits $n$ leading to ARE $1$, an exact expression for the MSE in this setting seems to be difficult to derive. The situation is even more involved in the adaptive encoding of setting (ii): an encoding and estimation strategy that is optimal for $n-1$ adaptive one-bit messages of a sample of size $n-1$, may not lead to a globally optimal strategy upon the recipient of the $n$th sample. Conversely, any one-step optimal strategy, in the sense that it finds the best one-bit message as a function of the current sample and the previous $2^{n-1}$ messages, is not guaranteed to be globally optimal. Our results imply that the ARE of any globally optimal strategy is $\pi/2$. 

\subsection*{Related Works}

As the variance $\sigma^2$ goes to zero, the task of finding $\theta$ using one-bit queries in the adaptive setting (ii) is easily solved by a bisection style method over the parameter space. Therefore, the general case of non-zero variance is a reminiscent of the noisy binary search problem with possibly infinite number of unreliable tests \cite{cicalese2002least, Karp:2007:NBS:1283383.1283478}. However, since we assume a continuous parameter space, a more closely related problem is that of one-bit analog-to-digital conversion of a noisy signal. For example, the sigma-delta modulator (SDM) analog-to-digital conversion \cite{1092194} uses one-bit threshold detector combined with a feedback loop to update an accumulated error state, and therefore falls under setting (ii). A SDM with a constant input $\theta$ corrupted by a Gaussian noise was studied in \cite{53738}, where it was shown that the output of the modulator converges to the true constant input almost surely. In other words, the SDM provides a consistent estimator for setting (ii). The rate of this convergence, however, was not analyzed and cannot be derived from the results of \cite{53738}. Our results imply that the rate of convergence of the MSE in SDM to a constant input is at most $\sigma^2\pi/2$ over the number of feedback iterations. \\

Our result of ARE of $\pi/2$ in the adaptive setting implies that even under coarse quantization constraints it is possible to achieve MSE in parametric estimation within only a relatively small penalty compared to the unconstrained estimator. A possible clue for this non-intuitive result is obtained from drawing the connection between our setting and the remote multiterminal source coding problem, also known as the CEO problem \cite{berger1996ceo, viswanathan1997quadratic, oohama1998rate, prabhakaran2004rate}. This connection, which is explained in details in Section~\ref{sec:ceo}, immediately leads to a lower bound of $4/3$ on the ARE in the distributed encoding of setting (iii). While this lower bound provides no new information compared to the lower bound of $\pi /2$ we derive later for setting (ii), it shows that the distributed nature of the problem is not a limiting factor in achieving MSE close to optimal even under one-bit quantization of each sample.\\

We also note that our settings (ii) and (iii) are special cases of 
\cite{zhang2013information} that consider adaptive and distributed estimation protocols for $m$ machines, each has access to $n/m$ independent samples. The main result of \cite{zhang2013information} are bounds on the estimation error as a function of the number of bits $R$ each machine uses for communication. The specialization of their result to our setting, by taking $m=n$ and $R=1$, leads to looser lower bounds then $\sigma^2\pi/(2n)$ for cases (ii) and (iii). The counterpart of our setting (iii) in the case of hypothesis testing was considered in \cite{52470}, although the results there cannot be extended to parametric estimation. Other related works include statistical inference under multiterminal data compression \cite{han1987hypothesis, zhang1988estimation}, and one-bit quantization constraints in compressed sensing  \cite{baraniuk2017exponential} and in MIMO detection in wireless communication \cite{singh2009limits}. \\

\subsection*{Paper Organization}

The rest of this paper is organized as follows: in Section~\ref{sec:problem} we define 
the main problem and notation. In Section~\ref{sec:ceo} we illustrate a connection between our parametric estimation problem and the remote multiterminal lossy compression problem. In Section~\ref{sec:sequential} we present our main results, deferring long proofs and technical results to the appendix. Concluding remarks are given in Section~\ref{sec:conclusions}. 

\section{Problem Formulation \label{sec:problem}}

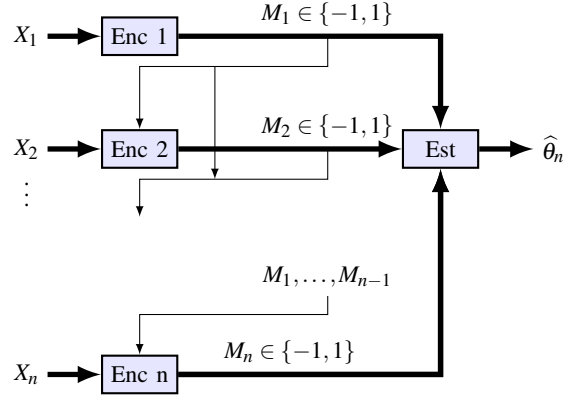
\begin{figure}
\begin{center}
\begin{tikzpicture}[node distance=2cm,auto,>=latex]
  \node at (0,0) (source) {$X_1$} ;
  \node[int1, right of = source, node distance = 1.5cm] (enc1) {Enc 1};  
\draw[->,line width = 2pt] (source) -- (enc1); 

 \node[below of = source, node distance = 1.5cm] (source2) {$X_2$};
\node[int1, right of = source2, node distance = 1.5cm] (enc2) {Enc 2};  
\draw[->,line width = 2pt] (source2) -- (enc2); 

\node[below of = source2, node distance = 3cm] (source3) {$X_n$};
\node[int1, right of = source3, node distance = 1.5cm] (enc3) {Enc n};  
\draw[->,line width = 2pt] (source3) -- (enc3);

\node[above of = source, node distance = 1cm] (dist) {$X_i \sim {\mathcal N} \left(\theta, \sigma^2 \right)$};
\node[below of = source2, node distance = 0.5cm] {$\vdots$};

\node[int1, right of = enc2, node distance = 4cm ] (est) {Est};
\draw[->,line width = 2pt] (enc1) -| node[above, xshift = -1.5cm] (mes1) { $M_1 \in \left\{-1,1\right\}$} (est);   
\draw[->,line width = 2pt] (enc2) -- node[above, xshift = 0.5cm] (mes2) { $M_2 \in \left\{-1,1\right\}$} (est);   
\draw[->] (mes1) -- +(0,-0.7) -| (enc2);

\draw[->] (mes2) -- +(0,-0.7) -| +(-2.5,-1.2);
\draw[->] (mes2)+(-1.5,0.8) -- +(-1.5,-0.7);
\draw[->,line width = 2pt] (enc3) -| (est);   

\node[below of = mes2, node distance = 2cm] (mes3) {$M_1,\ldots,M_{n-1} $};

\draw[->,line width = 2pt] (enc3) -| node[above, xshift = -2cm]  {$M_n \in \left\{-1,1\right\}$} (est);   

\draw[->] (mes3) -- +(0,-0.5) -| (enc3);

\node[right of = est, node distance = 1.5cm] (dest) {$\widehat{\theta}_n$};
\draw[->, line width=2pt] (est) -- (dest);
\end{tikzpicture}
\end{center}
\caption{\label{fig:sequential} Adaptive one-bit encoding: the $i$th encoder delivers a single bit message that is a function of its private sample $X_i$ and the previous $i-1$ messages $M_1,\ldots,M_{i-1}$.}
\end{figure}

Let $X_i$, $i=1,\ldots,n$, be $n$ independent samples from the normal distribution with mean $\theta$ and variance $\sigma^2$. 
We assume that the mean $\theta$ is drawn once from a prior distribution $\pi(\theta)$ on $\Theta$, which is a closed subset of the real line. We moreover assume that $\pi(\theta)$ is absolutely continuous with respect to the Lebesgue measure with density $\pi(d\theta)$. The problem we consider is the estimation of the parameter $\theta$ under the following constraints on the communication between the samples $X^n = (X_1,\ldots,X_n)$ and a centralized estimator: 
\begin{itemize}
\item[(i)] The estimator at time $n$ is only a function of the $n$ messages $M^n = \left(M_1,\ldots,M_n \right)$.
\item[(ii)] For each $i=1,\ldots,n$, the $i$th message $M_i$ is a function of the sample $X_i$ and the $i-1$ previous messages $M^{i-1}$.
\item[(iii)] The $i$th message $M_i$ takes only two possible values, say $1$ and $-1$. 
\end{itemize}
In other words, the $i$th message is defined by a function from the real line to $\{-1,1\}$ that is measurable with respect to the sigma algebra generated by $M^{i-1}$ and $X_i$, and the $n$ messages $M^n$ are the only available to the estimator. Upon observing $M^n$, the estimator produces an estimate $\widehat{\theta}_n(M^n)$ of $\theta$. A system describing the above scheme is illustrated in Fig.~\ref{fig:sequential}. \\

In this work we are concerned with the Bayes MSE risk defined as
\begin{equation}
\label{eq:error_def}
\mathbb E\left(\widehat{\theta}_n - \theta \right)^2,
\end{equation}
where the expectation is taken with respect to the distribution of $X^n$ and the prior distribution $\pi(\theta)$. \\

The main problem we consider is the minimization of \eqref{eq:error_def} over all encoding and estimation strategies and the characterization of its minimal value as a function of $n$. This minimization is the combination of the following two procedures: (1) selecting the $i$th message $M_i$ based on past messages and current observation $X_i$, and (2) estimating $\theta$ given messages $M^n$. 
We are interested in particular on the increase in sample complexity compared to the vanilla mean estimation without one-bit constraint. For this reason, we consider 
\begin{equation}
{\sigma^2} n \mathbb E\left(\widehat{\theta}_n - \theta \right)^2
\label{eq:relative_efficiency}
\end{equation}
in the limit as $n$ goes to infinity. Equation \eqref{eq:relative_efficiency} is the ratio between the MSE attained by the empirical mean of the samples and the MSE attained by the estimator $\widehat{\theta}_n$. Note that if the limit $\mathbb E\left(\widehat{\theta}_{n} - \theta \right)^2$ exists and finite, than \eqref{eq:relative_efficiency} is the ARE of $\widehat{\theta}_n$ \cite[Def. 6.6]{lehmann2006theory}. 


In addition to the notations defined above, we denote by $\phi(x)$ the standard normal density and by $\Phi(x)$ the standard normal cumulative distribution function. \\

Before deriving our main results, we comment on the relation between our setting and the remote multiterminal source coding problem, also known as the CEO problem.

\section{Relation to Remote Multiterminal Source Coding \label{sec:ceo}}
The setting of the CEO includes $n$ encoders, each has access to a noisy version of a random source sequence \cite{berger1996ceo}. 
The $i$th encoder observes $k$ noisy source symbols and transmit $R_i k$ bits to a central estimator. \par
Assuming that $\theta$ is drawn once from the prior $\pi(\theta)$, our mean estimation problem from one-bit samples under distributed encoding (setting (iii) in the Introduction) corresponds to the Gaussian CEO setting with $k=1$ source realization: the $i$th encoder uses $R_i=1$ bits to transmit a message that is a function of $X_i = \theta + \sigma Z_i$, where $Z_i$ is standard normal. As a result, a lower bound on the MSE distortion in estimating $\theta$ in the distributed encoding setting is given by the MSE in the optimal source coding scheme for the CEO with: $n$ terminals of codes rates $R_1 = \ldots = R_n = 1$, a Gaussian observation noise at each terminal of variance $\sigma^2$, and an arbitrary number of $k$ independent draws of $\theta$. Note that the difference between the CEO and ours lays in the privilege of each of the CEO encoders to describe $k$ realizations of $\theta$ using $k$ bits with MSE averaged over these realization, whereas our setting only allows $k=1$. 
 \\

By using an expression for the minimal MSE in the Gaussian CEO as the number of terminals goes to infinity, we conclude the following:
\begin{prop} \label{prop:ceo_lower_bound}
Assume that $\Theta = \mathbb R$ and that $\pi(\theta) = \mathcal N(0,\sigma_\theta^2)$. Then any estimator $\widehat{\theta}_n$ of $\theta$ in the distributed setting satisfies
\begin{equation} \label{eq:ceo_bound}
 n\mathbb E \left( \theta - \theta_n \right)^2 \geq \frac{4\sigma^2}{3} + O(n^{-1}),
\end{equation}
where the expectation is with respect to $\theta$ and $X^n$.
\end{prop}

\begin{proof}
We consider the expression \cite[Eq. 10]{chen2004upper} that gives the minimal distortion $D^\star$ in the CEO with $L$ observers and under a total sum-rate $R_\Sigma = R_1 + \ldots +R_L$:
\begin{equation} \label{eq:ceo_optimal_sumrate}
R_{\Sigma} = \frac{1}{2} \log^+ \left[ \frac{\sigma_\theta^2}{D^\star} \left( \frac{D^\star L}{ D^\star L - \sigma^2 + D^\star \sigma^2 / \sigma_\theta^2 }\right)^L  \right].
\end{equation}
Assuming $R_\Sigma = n$ and $L=n$, we get
\begin{equation} \label{eq:ceo_optimal_sumrate}
n = \frac{1}{2} \log_2 \left[ \frac{\sigma_\theta^2}{D^\star} \left(\frac{ D^\star n }{D^\star n - \sigma^2 + D^\star \sigma^2/\sigma_\theta^2 }  \right)^n  \right].
\end{equation}
The value of $D^\star$ that satisfies the equation above describes the MSE under an optimal allocation of the sum-rate $R_\Sigma = n$ among the $n$ encoders. Therefore, $D^\star$ provides a lower bound to the CEO distortion with $R_1=\ldots,R_n = 1$ and hence a lower bound to the minimal MSE in estimating $\theta$ in the distributed encoding setting. By considering $D^\star$ in \eqref{eq:ceo_optimal_sumrate} as $n\rightarrow \infty$, we see that 
\[
D^\star = \frac{ 4\sigma^2 }{3n + 4 \sigma^2 / \sigma_\theta^2 } + o(n^{-1}) =  \frac{4\sigma^2}{3n} + o(n^{-1}). 
\]
\end{proof}

Prop.~\ref{prop:ceo_lower_bound} implies that, unlike in the centralized setting, there is a loss in efficiency in estimating $\theta$ due to one-bit measurements in this setting. In the next section we show that the ARE in adaptive encoding setting does not exceeds $\pi/2$, and thus provides a tighter lower bound for the distributed encoding setting than $4/3$ of  \eqref{eq:ceo_bound}. \\

We note although the lower bound \eqref{eq:ceo_bound} was derived assuming the optimal allocation of $n$ bits per observation among the encoders, this bound cannot be tightened by considering the CEO distortion while enforcing the condition $R_1=\ldots = R_n = 1$. Indeed, an upper bound for the CEO distortion under the condition $R_1=\ldots = R_n = 1$ follows from \cite{KipnisRini2017}, and leads to
\[
D_{CEO} \leq  \left( \frac{1}{\sigma_\theta^2} +  \frac{3n}{4\sigma^2 + \sigma_\theta^2} \right)^{-1}   =
\frac{4 \sigma^2}{3n} +  \frac{\sigma_\theta^2}{3n} + O(n^{-2}),
\]
which is equivalent to \eqref{eq:ceo_bound} when $\sigma_\theta$ goes to zero. 

\section{Results \label{sec:sequential}}
The first main result of this paper, as described in Thm.~\ref{thm:adpative_lower_bound} below, states that the ARE of any adaptive estimator cannot be lower than $\pi/2$. Next, we provide a particular adaptive estimation scheme and show in Thm.~\ref{thm:sgd} that its efficiency is $\pi/2$. Finally, in Thm.~\ref{thm:opt_one_step}, we provide an adaptive estimation scheme that is one-step optimal in the sense that at each step $i$, the encoder send the message $M_i^\star$ that minimizes the MSE given $X_i$ and the previous $M^{i-1}$ messages. While it is not clear whether the efficiency of this last scheme is $\pi/2$, numerical simulations suggests that the MSE of this scheme times $n$ also converges to $\pi/2$ faster than the first scheme.

%
%
%
%
%


\subsection{A lower bound on adaptive one-bit schemes}
Our first results asserts that the ARE \eqref{eq:relative_efficiency} of any adaptive estimation scheme is bounded from below by $\pi/2$, as follows from the following theorem:
\begin{thm}[minimal relative effeciency] \label{thm:adpative_lower_bound}
Let $\widehat{\theta}_n$ be any estimator of $\theta$ in the adaptive setting of Fig.~\ref{fig:sequential}. Assume that the density of the prior $\pi(\theta)$ converges to zero at the endpoints of the interval $\Theta$. Then
\[
\mathbb E\left[ (\theta-\theta_n)^2 \right] \geq  \frac{\pi \sigma^2 }{2n +\pi \sigma^2  I_0}  =  \frac{\pi}{2n}\sigma^2+O(n^{-2}),
\]
where 
\[
I_0 = \mathbb E \left( \frac{d}{d\theta} \log \pi (\theta) \right)^2
\]
is the Fisher information with respect to a location model in $\theta$. 
\end{thm}

\subsubsection*{Sketch of Proof}
The main idea in the proof is to bound from above the Fisher information of any set of $n$ single-bit messages with respect to $\theta$. Once this bound is achieved, the result follows by using the van-Trees inequality \cite[Thm. 2.13]{tsybakov2008introduction},\cite{gill1995applications} which bounds from below the MSE of any estimator of $\theta$ by the inverse of the expected value of the aforementioned Fisher information plus $I_0$. The details are in the Appendix.\\

Next, we present an adaptive estimation scheme that attains ARE of $\pi/2$. 
\subsection{Asymptotically optimal estimator}
Let $\left\{\gamma_n \right\}_{n=1}^\infty$ be a strictly positive sequence satisfying:
\begin{equation} \label{eq:conditions}
\begin{cases}
\frac{\gamma_n - \gamma_{n+1}}{\gamma_n} = o(\gamma_n), &  \\
\sum_{n=1}^\infty \frac{\gamma_n^{(1+\lambda)/2}} {\sqrt{n}} < \infty, & 
\mathrm{for~some~}0< \lambda \leq 1
\end{cases}
\end{equation}
(e.g., $\gamma_n = n^{-\beta}$ for $\beta \in (0,1)$). Consider the following estimator $\widehat{\theta}_n$ for $\theta$:  
\begin{equation}
\label{eq:sgd_alg}
\theta_n = \theta_{n-1} +  \gamma_n \sgn (X_n - \theta_{n-1}), \quad n = 1,2,\ldots,
\end{equation}
and set the $n$th step estimation as
\begin{equation} \label{eq:sgd_est}
\widehat{\theta}_n =  \frac{1}{n} \sum_{i=1}^n  \theta_i. 
\end{equation}

For the estimator defined by \eqref{eq:sgd_alg} and \eqref{eq:sgd_est} we have the following results:
\begin{thm} \label{thm:sgd}
The sequence $\widehat{\theta}_n$ of \eqref{eq:sgd_est} satisfies
\begin{enumerate}
\item[(i)]
\[
\sqrt{n} \left( \widehat{\theta}_n - \theta \right) \overset{d}{\rightarrow} \mathcal N \left(0,  \pi \sigma^2 /2 \right).
\]
\item[(ii)] In addition to the conditions above, assume that $\gamma_n = o(n^{-2/3})$ and $\sum_{n=1}^\infty \gamma_n = \infty$ (e.g., $\gamma_n = n^{-\beta}$ with $2/3<\beta<1$). Then
\[
\lim_{n\rightarrow \infty} n\mathbb E \left[ \left(\theta-\widehat{\theta}_n \right)^2 \right] = \frac{\pi}{2} \sigma^2 . 
\]
\end{enumerate}

\end{thm}

\subsubsection*{Proof}
The asymptotic behavior of \eqref{eq:sgd_est} is a special case of \cite[Thm. 4]{polyak1992acceleration} and \cite[Thm. 2]{polyak1990new}. The details are in the Appendix.\\

Thm.~\ref{thm:sgd} implies that the estimator $\widehat{\theta}_n$, defined by \eqref{eq:sgd_est} and \eqref{eq:sgd_alg}, attains the minimal ARE as established by Thm.~\ref{thm:adpative_lower_bound}.
\par
Note that $\theta_0$ is not explicitly defined 
in equation \eqref{eq:sgd_est}. While a reasonable initialization is $\theta_0 = \mathbb E [\theta]$, Thm.~\ref{thm:sgd} implies that the asymptotic behavior of the estimator is indifferent to this initialization. Thus, the optimal efficiency is attained regardless of the prior distribution on $\theta$ or the radius of the parameter space $\Theta$. Nevertheless, the bound in Thm.~\ref{thm:adpative_lower_bound} suggests that the non-asymptotic estimation error can be significantly reduced whenever the location information $I_0$ is large. In contrast, the one-step optimal scheme presented in the following subsection updates the prior distribution on $\theta$ given all information gathered until step $n-1$ to provide the step $n$ estimate and prior. In particular, the this scheme exploit the prior information on $\theta$ provided by $\pi(\theta)$. 

\subsection{One-step optimal estimation}
We now consider an estimation scheme that posses the property of \emph{one-step optimality}: at each step $i$, the $i$th encoder designs the detection region $M_i^{-1}(1)$ such that the MSE given $M^i$ is minimal. In other word, this scheme designs the messages $M^n$ in a greedy manner, such that the MSE at step $i$ is minimal given the current state of the estimation described by $M^{i-1}$. \\

The following theorem determine the structure of the message that minimizes the next step MSE:
\begin{thm}[optimal one-step estimation] \label{thm:opt_one_step}
Let $\pi(\theta)$ be an absolutely continuous log-concave probability distribution. Given a sample $X$ from the distribution $\mathcal N(\theta, \sigma^2)$, define 
\begin{equation}
\label{eq:adaptive_main_message}
M = \sgn(X - \tau),
\end{equation}
where $\tau$ satisfies the equation
\begin{equation}
 \label{eq:fixed_point}
 \tau = \frac{m^-(\tau) + m^+(\tau)}{2},
\end{equation}
with
\begin{align*}
m^-(\tau)  & = \frac{\int_{-\infty}^{\tau} \theta \pi(d\theta) }{\int_{-\infty}^{\tau} \pi(d\theta)} ,\\
m^+(\tau) & = \frac{\int_{\tau}^\infty \theta \pi(d\theta) }{\int_{\tau}^\infty \pi(d\theta)} .
\end{align*}
Then for any estimator $\widehat{\theta}$ which is a function of $M'(X) \in \{-1,1\}$, we have
\begin{equation}
\label{eq:opt_cond}
\mathbb E \left(\theta-\widehat{\theta}(M')\right)^2 \geq  \mathbb E \left(\theta- \mathbb E[\theta|M]\right)^2,
\end{equation}
\end{thm}

\begin{proof}
The proof is completed by the following two lemmas, proofs of which can be found in the Appendix:
\begin{lem} \label{lem:unique}
Let $f(x)$ be a log-concave probability density function. Then the equation 
\begin{equation}
\label{eq:lem_fixed_point}
2x = \frac{\int_x^\infty uf(u)du}{\int_x^\infty f(u)du} + \frac{\int_{-\infty}^x uf(u)du}{\int_{-\infty}^x f(u)du} 
\end{equation}
has a unique solution.
\end{lem}
\begin{lem} \label{lem:adaptive}
Let $U$ be a random variable with probability density function  $P(du)$. Then the one-bit message $M^\star\in \{-1,1\}$ that minimizes
\[
\int \left( u - \mathbb E[U|M(u)]  \right)^2 P(du)
\]
is given by
\[
M^\star  =  \sgn(U - \tau),
\]
where $\tau$ is the unique solution to
 \[
2 \tau = \frac{\int_{\tau}^\infty u P(du)} {\int_{\tau}^\infty P(du)} + \frac{\int_{-\infty}^{\tau} u P(du)}{\int_{-\infty}^{\tau} P(du)}.
\]
\end{lem}
\end{proof}

Thm.~\ref{thm:opt_one_step} suggests the following adaptive encoding and estimation scheme: 
\begin{itemize}
\item Initialization: set $P_0(t) = \pi(\theta)$.
\item For $n\geq 1$:
\begin{enumerate}
\item Update the prior as
\begin{align}
P_n(t) = & P(\theta=t |M^n) \\
& = \frac{ P\left( \theta=t | M^{n-1} \right) P(M_n | \theta = t , M^{n-1})  } { P(M_n | M^{n-1} )} \nonumber \\ 
& = \alpha_n  P_{n-1}(t) \Phi\left(M_n \frac{ t - \tau_{n-1} }{\sigma} \right), \label{eq:density_update}
\end{align}
where $\alpha_n$ is a normalization coefficient that equals
\[
\alpha_n = \left(\int_{\mathbb R} P_{n-1}(t) \Phi\left(M_n \frac{t- \tau_{n-1} }{\sigma} \right)  dt \right)^{-1}. 
\]
\item The $n$th estimate for $\theta$ is the conditional expectation of $\theta$ given $M^n$, namely
\begin{equation}
\theta_n = \mathbb E \left[ \theta| M^n\right] = \int_{-\infty}^\infty t P_n(t) dt. \label{eq:estimator_update}
\end{equation}
\item Solve equation \eqref{eq:fixed_point} with the updated prior $P_n(t)$ instead of $\pi(d\theta)$. Note that since the standard normal cdf $\Phi(x)$ is log-concave, the updated prior $P_n(t)$ remains log-concave and thus a unique solution to \eqref{eq:fixed_point} is guaranteed by Lem.~\ref{lem:unique}. 
\item Update the $(n+1)$th message as
\begin{equation}\label{eq:message_update}
M_{n+1} = \sgn(X_{n+1}-\tau_n)
\end{equation}
\end{enumerate}
\end{itemize}
Since equation \eqref{eq:fixed_point} has no analytic solution, it is hard to derive the asymptotic behavior of the estimator defined by \eqref{eq:estimator_update} and \eqref{eq:message_update}. We conjecture that it attains the asymptotic relative efficiency of $\sigma^2\pi/2$ as can be observed from the numerical simulation illustrated in Fig.~\ref{fig:adaptive_error}. Also shown in Fig.~\ref{fig:adaptive_error} are the normalized MSE of the asymptotically optimal estimator defined by \eqref{eq:sgd_alg} and \eqref{eq:sgd_est}, as well as the MSE achieved by the empirical mean of the samples for the same sample realization.

\begin{figure}
\begin{center}
\begin{tikzpicture}
\node at (0,0) {\includegraphics[scale=0.4]{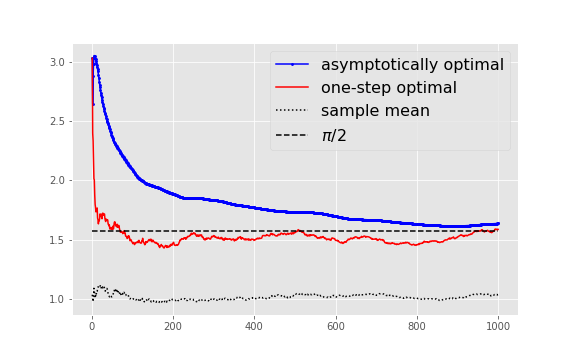}};
\node[rotate = 90, scale = 0.7] at (-3.8,0) {$n \mathbb E \left(\widehat{\theta}_n - \theta \right)^2$};
\node[scale = 0.7] at (0,-2.4) {$n$};
\end{tikzpicture}
\caption{Normalized empirical risk $n\left(\widehat{\theta}_n-\theta\right)^2$ versus number of samples $n$ for $500$ Monte Carlo trials. In each trial, $\theta$ is chosen uniformly in the interval $(-3,3)$. \label{fig:adaptive_error}  }
\end{center}
\end{figure}

\section{Conclusions \label{sec:conclusions}}
We considered the MSE risk and asymptotic relative efficiency in estimating the mean of a normal distribution from a single-bit encoding of each sample from this distribution. In the adaptive scenario where each one-bit message is a function of the previously seen messages and current sample, we showed that the minimal ARE is $\pi/2$. Namely, there is a penalty factor of at least $\pi/2$ on the asymptotic MSE risk in estimating the mean compared to an estimator that has full access to the sample. We also showed that this lower bound is tight by presenting an adaptive estimation procedure that attains it. The lower bound of $\pi/2$ on the ARE also holds in the fully distributed case where each single-bit message is only a function of a single independent sample, although the question whether this ARE is achievable remains still open. 



\appendix

\section{Proofs}
In this appendix we provide detailed proofs of our main results as described in Section~\ref{sec:sequential}.

\subsection*{Proof of Thm.~\ref{thm:adpative_lower_bound}}

We first prove the following two lemmas:
\begin{lem} \label{lem:bound_intervals}
For any $x_1 \geq \ldots \geq x_n \in \mathbb R$, we have 
\begin{equation}
\frac{ \left(  \sum_{k=1}^n (-1)^{k+1}\phi(x_k) \right)^2} 
{\left( \sum_{k=1}^n (-1)^{k+1} \Phi(x_k) \right)\left(1- \sum_{k=1}^n (-1)^{k+1} \Phi(x_k) \right)  } \leq \frac{2} {\pi}. \label{eq:bound_intervals}
\end{equation}
\end{lem}
\begin{lem} \label{lem:fisher_bound}
Let $X\sim \mathcal N(\theta,\sigma^2)$ and assume that 
\[
M(X) = \begin{cases} 1,& X \in A, \\
-1, & X \notin A.
\end{cases}
\]
Then the Fisher information of $M$ with respect to $\theta$ is bounded from above by $2/(\pi \sigma^2)$.
\end{lem}

\subsubsection*{Proof of Lem.~\ref{lem:bound_intervals}}
We use induction on $n \in \mathbb N$. For the base case $n=1$ we have 
\begin{equation} \label{eq:induction_base}
\frac{  \phi^2(x)} 
{\Phi(x) \left(1 - \Phi(x) \right) },
\end{equation}
which is the weight function in the probit analysis and is known to be a strictly decreasing function of $|x|$ \cite{Samford1953}. 
In particular, the maximum of \eqref{eq:induction_base} is obtained at $x=0$ where it equals $2/\pi$. \par
Assume now that \eqref{eq:bound_intervals} holds for all integers up to some $n = N-1$ and consider the case $n = N$. The maximal value of \eqref{eq:bound_intervals} is attained for the same $(x_1,\ldots,x_N) \in \mathbb R^N$ that attains the maximal value of 
\begin{align*}
& g(x_1,\ldots, x_N) \triangleq 2 \log \left(  \sum_{k=1}^{N} (-1)^{k+1} \phi(x_k) \right) - \\
&  \log
\left( \sum_{k=1}^N (-1)^{k+1} \Phi(x_k) \right)
-\log \left(1 -  \sum_{k=1}^N (-1)^{k+1} \Phi(x_k) \right) \\
& = 2 \log \delta_N - \log \Delta_N - \log \left(1 - \Delta_N  \right),
\end{align*}
where we denoted $\delta_N \triangleq \sum_{k=1}^{N} (-1)^{k+1} \phi(x_k)$ and $\Delta_N =  \sum_{k=1}^N (-1)^{k+1} \Phi(x_k)$. The derivative of $g(x_1,\ldots,x_N)$ with respect to $x_k$ is given by
\[
\frac{\partial  g}{\partial x_k} = \frac{2 (-1)^{k+1} \phi'(x_k)}{\delta_N} -\frac{(-1)^{k+1} \phi(x_k)}{\Delta_N } + \frac{(-1)^{k+1} \phi(x_k)}{1-\Delta_N }.
\]
Using the fact that $\phi'(x) = -x \phi(x)$, we conclude that the gradient of $g$ vanishes only if
\[
x_k = \frac{\delta_N}{2} \left( \frac{1}{\Delta_N} - \frac{1}{1-\Delta_N} \right),\quad k=1,\ldots,N.
\]
In particular, the condition above implies $x_1 = \ldots = x_N$. If $N$ is odd then for $x_1=\ldots =x_N$ we have that the LHS of \eqref{eq:bound_intervals} equals
\[
\frac{\phi(x_1)^2}{ \Phi(x_1) (1 - \Phi(x_1))},
\]
which was shown to be not larger than $2/pi$. If $N$ is even, then for any constant $c$ the limit of the LHS of \eqref{eq:bound_intervals} as $(x_1,\ldots,x_N)\rightarrow (c,\ldots,c)$ exists and equals zero. Therefore, the maximum of the LHS of \eqref{eq:bound_intervals} is not attained at the line $x_1=\ldots=x_N)$. We now consider the possibility that the LHS of \eqref{eq:bound_intervals} is maximized at the borders, as one or more of the coordinates of $(x_1,\ldots,x_N)$ approaches plus or minus infinity. For simplicity we only consider the cases where $x_N$ goes to minus infinity or $x_1$ goes to plus infinity (the general case where the first $m$ coordinates goes to infinity or the last $m$ to minus infinity is obtained using similar arguments). Assume first $x_N \rightarrow -\infty$. Then the LHS of \eqref{eq:bound_intervals} equals
\begin{align*}
\frac{ \left(  \sum_{k=1}^{N-1} (-1)^{k+1}\phi(x_k) \right)^2} 
{\left( \sum_{k=1}^{N-1} (-1)^{k+1} \Phi(x_k) \right)\left(1- \sum_{k=1}^{N-1} (-1)^{k+1} \Phi(x_k)  \right) } ,
\end{align*}
which is smaller than $2/\pi$ by the induction hypothesis. Assume now that $x_1 \rightarrow \infty$. Then the LHS of
\eqref{eq:bound_intervals} equals
\begin{align*}
& \frac{ \left(  \sum_{k=2}^{N} (-1)^{k+1}\phi(x_k) \right)^2} 
{\left( 1 + \sum_{k=2}^{N} (-1)^{k+1} \Phi(x_k) \right)\left(1- 1 - \sum_{k=2}^{N} (-1)^{k+1} \Phi(x_k)  \right) }  \\
& = \frac{ \left(  -\sum_{m=1}^{N} (-1)^{m+1}\phi(x'_m) \right)^2} 
{\left( 1 - \sum_{m=1}^{N-1} (-1)^{m+1} \Phi(x'_{m}) \right)\left( \sum_{m=1}^{N-1} (-1)^{m+1} \Phi(x'_{m})  \right) },
\end{align*}
where $x'_{m} = x_{m+1}$. The last expression is also smaller than $2/\pi$ by the induction hypothesis. This proves Lem.~\ref{lem:bound_intervals}. \\

\subsubsection*{Proof of Lem.~\ref{lem:fisher_bound}}
The Fisher information of $M$ with respect to $\theta$ is given by
\begin{align}
I_\theta & =  \mathbb E \left[ \left( \frac{d}{d\theta} \log P\left( M | \theta \right) \right)^2 |\theta \right] \nonumber \\
& = \frac{ \left(\frac{d}{d\theta} P(M=1|\theta) \right)^2}{P(M=1| \theta)} + \frac{ \left(\frac{d}{d\theta} P(M=-1|\theta) \right)^2} {P(M=-1| \theta)} \nonumber \\
& =  \frac{ \left( \frac{d}{d\theta} \int_A \phi \left( \frac{x-\theta}{\sigma} \right)dx \right)^2} { P(M=1| \theta) } + \frac{ \left( \frac{d}{d\theta}\int_A \phi \left( \frac{x-\theta}{\sigma} \right)dx \right)^2} { P(M=-1| \theta) } \nonumber \\ 
& \overset{(a)}{=} \frac{ \left( - \int_A \phi' \left( \frac{x-\theta}{\sigma} \right)dx \right)^2} {\sigma^2 P(M=1| \theta) } + \frac{ \left(- \int_A \phi' \left( \frac{x-\theta}{\sigma} \right)dx \right)^2} { \sigma^2P(M=-1| \theta) } \nonumber \\ 
& = \frac{\left( \int_A \phi'\left( \frac{x-\theta}{\sigma} \right) dx \right)^2 }{  \sigma^2 P(M=1 | \theta) \left(1-P(M=1|\theta) \right)  }, \nonumber \\
& = \frac{\left( \int_A \phi'\left( \frac{x-\theta}{\sigma} \right) dx \right) \left( \int_A \phi'\left( \frac{x-\theta}{\sigma} \right) dx \right)}{  \sigma^2 \left( \int_A \phi \left( \frac{x-\theta}{\sigma} \right) dx \right)  \left(1- \int_A \phi \left( \frac{x-\theta}{\sigma} \right) dx \right) }, \label{eq:lem_fisher_bound_proof1}
\end{align}
where differentiation under the integral sign in $(a)$ is possible since $\phi(x)$ is differentiable with absolutely integrable derivative $\phi'(x) = -x\phi(x)$. Regularity of the Lebesgue measure implies that for any $\epsilon>0$, there exists a finite number $k$ of disjoint open intervals $I_1,\ldots I_k$ such that 
\[
\int_{A\setminus \cup_{j=1}^k I_j }  dx < \epsilon \sigma^2,
\]
which implies that for any $\epsilon' > 0$, the set $A$ in \eqref{eq:lem_fisher_bound_proof1} can be replaced by a finite union of disjoint intervals without increasing $I_\theta$ by more than $\epsilon'$. It is therefore enough to proceed in the proof assuming that $A$ is of the form
\[
A = \cup_{j=1}^k (a_j,b_j),
\]
with $\infty \leq a_1 \leq \ldots a_k$, $b_1 \leq b_k \leq \infty$ and $a_j \leq b_j$ for $j=1,\ldots,k$. Under this assumption we have
\begin{align*}
\mathbb P(M_n=1| \theta) & = \sum_{j=1}^k \mathbb P\left(X_n \in (a_j,b_j) \right)  \\
& = \sum_{j=1}^k \left( \Phi \left(\frac{b_j-\theta}{\sigma} \right) -  \Phi \left(\frac{a_j-\theta}{\sigma} \right)  \right),
\end{align*}
so \eqref{eq:lem_fisher_bound_proof1} can be rewritten as
\begin{align}
& =   \frac { \left( \sum_{j=1}^{k} \phi \left(\frac{a_j-\theta}{\sigma} \right) - \phi \left( \frac{b_j-\theta} {\sigma} \right)  \right)^2 } 
{\sigma^2 \left( \sum_{j=1}^k \Phi \left( \frac{b_j-\theta }{\sigma}\right) - \Phi \left( \frac{a_j-\theta }{\sigma}\right)  \right) }  \nonumber \\
& \times \frac {1} 
{1- \left( \sum_{j=1}^k \Phi \left( \frac{b_j-\theta }{\sigma}\right) - \Phi \left( \frac{a_j-\theta }{\sigma}\right)  \right) } 
\label{eq:lemma_J}
\end{align}
It follows from Lem.~\ref{lem:bound_intervals} that for any $\theta \in \mathbb R$ and any choice of the intervals endpoints, \eqref{eq:lemma_J} is smaller than $2/(\sigma^2 \pi)$. Therefore, the proof of  Lem.~\ref{lem:fisher_bound} is now completed.
 \\

We now consider the proof of Thm.~\ref{thm:adpative_lower_bound}. In order to bound from above the Fisher information of any set of $n$ single-bit messages with respect to $\theta$, we first note that, without loss of generality, each message $M_i$ can be written in the form
\begin{equation}
\label{eq:general_messages}
M_i = \begin{cases}
X_i \in A_i & 1, \\
X_i \notin A_i & -1,
\end{cases} 
\end{equation}
where $A_i \subset \mathbb R$ is a Lebesgue measurable set. Indeed, any measurable function $M(X_i) \in \{-1,1\}$ can be written in the form \eqref{eq:general_messages} with $A_i = M^{-1}(1)$. Consider the conditional distribution $P({M^n|\theta})$ of $M^n$ given $\theta$. We have 
\begin{align}
P\left( M^n | \theta \right) & =  \prod_{i=1}^n P\left(M_i | \theta, M^{i-1} \right), \label{eq:adpt_lower_bound_proof:1}
\end{align}
where $P\left(M_i =1 | \theta, M^{i-1}  \right) = \mathbb P\left( X_i \in A_i\right)$. The Fisher information of $M^n$ with respect to $\theta$ is given by 
\begin{align}
I_\theta(M^n) = -\mathbb E \left[ \frac{d^2}{d\theta^2} \log P(M^n | \theta) \right] = \sum_{i=1}^n I_\theta (M_i|M^{i-1}),
\label{eq:fisher_information}
\end{align}
where 
\[
I_\theta (M_i|M^{i-1}) = \mathbb E \left( \frac{d}{d\theta} \log P(M_i | \theta, M^{i-1}) \right)^2
\]
 is the Fisher information of the distribution of $M_i$ given $M^{i-1}$, where it follows from Lem.~\ref{lem:fisher_bound} that $I_\theta (M_i|M^{i-1}) \leq 2/(\pi \sigma^2)$. We now use the following theorem from \cite[Thm. 2.13]{tsybakov2008introduction} (see also \cite{van2004detection, gill1995applications}):

\begin{thm}[The van Trees inequality \cite{tsybakov2008introduction}] \label{thm:vanTrees}  Denote by $p(\cdot,\theta)$ the density of $P_\theta$ with respect to the Lenesgue measure. Assume that: (i) the density $p(x,\theta)$ is measureable in $(x,\theta)$ and absolutely continuous in $t$ for almost all $x$ with respect to the Lebesgue measure. (ii) The Fisher information
\[
I(\theta) = \int \left( \frac{p'(x,\theta)}{p(x,\theta)} \right)^2 p(x,\theta) dx,
\]
where $p'(x,t)$ denotes the derivative of $p(x,\theta)$ in $t$, is finite and integrable on $\Theta$. (iii) The prior density $\pi(\theta)$ is absolutely continuous on its support $\Theta$ with zero mass at the boundries of $\Theta$, and has a finite Fisher information
\[
I_0 = \int_{\Theta} \frac{\left( \pi'(\theta) \right)^2} {\pi(\theta)} d\theta.
\]
Then, for any estimator $\widehat{t}(\mathbf X)$, the Bayes risk is bounded as follows:
\[
\int_{\theta} \mathbb E \left[ \left( \widehat{t}(\mathbf X) - \theta\right)^2 \right] \pi(d\theta) \geq \frac{1}{ \int I(\theta) \pi(d \theta) + I_0 }.
\]
\end{thm}
Thm.~\ref{thm:vanTrees} applied to our problem with $p(x,\theta) = P(M^n|\theta)$ implies
\begin{align*}
\mathbb E \left(\widehat{\theta}_n - \theta \right)^2 &  \geq \frac{1}{ \mathbb E I_\theta(M^n) + I_0} \\
& = \frac{1}{ \sum_{i=1}^n I_\theta (M_i | M^{i-1} ) + I_0} \\
& \geq \frac{1}{ 2n/(\pi \sigma^2) + I_0}.
\end{align*}

\QEDA

\subsection*{Proof of Thm.~\ref{thm:sgd}}
The algorithm given in \eqref{eq:sgd_alg} and \eqref{eq:sgd_est} is a special case of a more general class of estimation procedures given in \cite{polyak1992acceleration} and \cite{polyak1990new}. Specifically, (i) in Thm.~\ref{thm:sgd} follows directly from the following simplified version of \cite[Thm. 4]{polyak1992acceleration}:
\begin{thm}{\cite[Thm. 4]{polyak1992acceleration}} \label{thm:polyak_juditsky}
Let 
\[
X_i = \theta + Z_i,\quad i=1,\ldots,n,
\]
where the $Z_i$s are i.i.d. with zero means and finite variances. Define
\begin{align*}
\theta_i & = \theta_{i-1} + \gamma_i \varphi(X_i - \theta_{i-1}), \\
\widehat{\theta}_n & = \frac{1}{n} \sum_{i=0}^{n-1} \theta_i, 
\end{align*}
where in addition, assume the following: 
\begin{enumerate}
\item[(i)] There exits $K_1$ such that $\left| \varphi(x) \right| \leq K_1(1+|x|)$ for all $x\in \in \mathbb R$.
\item[(ii)] The sequence $\left\{ \gamma_i \right\}_{i=1}^\infty$ satisfies conditions \eqref{eq:conditions}.
\item[(iii)] The function $\psi(x) \triangleq \mathbb E \varphi(x+Z_1)$ is differentiable at zero with $\psi'(0)>0$, and satisfies $\psi(0)=0$ and $x\psi(x) >0$ for all $x\neq 0$.
Moreover, assume that there exists $K_2$ and $0<\lambda \leq 1$ such that
\begin{equation}
\label{eq:Polyak_Juditsky_cond3}
\left| \psi(x) - \psi'(0)x \right|\leq K_2 |x|^{1+\lambda}.
\end{equation}
\item[(iv)] The function 
$\chi(x) \triangleq \mathbb E \varphi^2(x+Z_1)$ is continuous at zero. 
\end{enumerate}
Then $\widehat{\theta}_n \rightarrow \theta$ almost surely and $ \sqrt{n}(\widehat{\theta}_n - \theta)$ converges in distribution to $\mathcal N(0,V)$, where
\[
V = \frac{ \chi(0)} {\psi'^2(0)}. 
\]
\end{thm}

Using the notation above, we set $\varphi(x) = \sgn(x)$ and $Z_i = X_i - \theta$. We have that $\chi(x) = \mathbb E \sgn^2(x+Z_1) = 1$, so $\chi(0) = 1$. In addition,
\begin{align*}
\psi(x) & = \mathbb E \sgn(x+ Z_1) = \int_{-\infty}^\infty \sgn(x+z) \frac{1}{\sqrt{2\pi}\sigma} e^{-\frac{z^2}{2\sigma^2}} dz \\
& = \int_{-x}^\infty \frac{1}{\sqrt{2\pi}\sigma} e^{-\frac{z^2}{2\sigma^2}} dz -\int_{-\infty}^{-x} \frac{1}{\sqrt{2\pi}\sigma} e^{-\frac{z^2}{2\sigma^2}} dz.
\end{align*}
This leads to 
\begin{align*}
\psi'(x) & = \frac{1}{\sqrt{2\pi}\sigma} e^{-\frac{x^2}{2\sigma^2}} dz +\frac{1}{\sqrt{2\pi}\sigma} e^{-\frac{z^2}{2\sigma^2}} dz,
\end{align*}
so $\psi'(0) = \frac{2}{\sqrt{2\pi}\sigma}$. It is now easy to verify that the rest of the conditions in Thm.~\ref{thm:polyak_juditsky} are fulfilled for any $\lambda > 0$. Since 
\[
\frac{\chi(0)}{\psi'^2(0)} = \frac{\pi \sigma^2}{2},
\]
Thm.~\ref{thm:sgd}-(i) follows from Thm.~\ref{thm:polyak_juditsky}.  \\
In order to prove Thm.~\ref{thm:sgd}-(ii), we consider:
\begin{thm}{ \cite[Thm. 2]{polyak1990new}} \label{thm:polyak_new}
Let
\begin{align} \label{eq:polyak_new_measurements}
\begin{cases}
U_n = U_{n-1} - \gamma_n \varphi(Y_n), & Y_n = f'(U_{n-1})+Z_n \\
\bar{U}_n= \frac{1}{n} \sum_{i=1}^n U_n, & n=1,2,\ldots.
\end{cases}
\end{align}
Assume that the function $f(x)$ is twice differentiable with a strictly positive and uniformly bounded second derivative. In particular, $f(x)$ is convex with a unique minimizer $x^\star \in \mathbb R$. Moreover, assume that the noises $Z_n$ are uncorrelated and identically distributed with a distribution for which the Fisher information exits. Let $\psi(x)$ and $\chi(x)$ be defined as in  Thm.~\ref{thm:polyak_juditsky}-(iii) and satisfies the conditions there. Assume in addition that $\chi(0)>0$, condition \eqref{eq:Polyak_Juditsky_cond3} with $\lambda = 1$, 
and there exits $K_3$ such that 
\[
\mathbb E \left[ | \varphi(x+Z_1) |^4 \right] \leq K_3(1+|x|^4). 
\]
Finally, assume that the sequence $\{\gamma_n \}$ satisfies conditions \eqref{eq:conditions} and the additional conditions in Thm.~\ref{thm:sgd}-(ii). Then
\[
V_n \triangleq \mathbb E \left[ \left(\bar{U}_n-x^\star \right)^2 \right] = n^{-1}\frac{\chi(0)} { (\psi'(0))^2 (f''(x^\star))^2 } + o(n^{-1}).
\]
\end{thm}

We now use Thm.~\ref{thm:polyak_new} with $f(x) = 0.5(x-\theta)^2$, $\varphi(x) = -\sgn(-x)$, $Z_n = \theta-X_n$ and $U_n = \theta_n$. From \eqref{eq:polyak_new_measurements} we have
\begin{align*} 
\theta_n & = \theta_{n-1} + \gamma_n \sgn(\theta-\theta_{n-1} - Z_n )  \\
& = \theta_{n-1} + \gamma_n \sgn(X_n-\theta_{n-1} ),
\end{align*}
so the estimator $\widehat{\theta}_n$ defined by $\widehat{\theta}_n$ equals to the one defined by \eqref{eq:sgd_est} and \eqref{eq:sgd_alg}. Note that
\[
\mathbb E \left[ | \varphi(x+Z_1) |^4 \right] = 1 \leq K_3(1+|x|^4)
\]
for any $K_3\geq 1$, the Fisher information of $Z_1$ is $\sigma^2$, $\chi(x) = 1 > 0$, and that 
the conditions in Thm.~\ref{thm:polyak_new} on $\psi(x)$ and $\chi(x)$ were verified to hold in the first part of the proof. In particular, $\psi'(0) = \sqrt{2}/ (\sqrt{\pi}/\sigma)$. Since $f(x)$ satisfies the conditions above with $x^\star = \theta$ and $f''(x) = 1$. Thm.~\ref{thm:polyak_new} implies 
\[
n V_n = \mathbb E \left[ \left(\widehat{\theta}_n-\theta \right)^2 \right]  = \frac{\pi}{2}  + o(1).
\]

 \QEDA

\subsection*{Proof of Thm.~\ref{thm:opt_one_step}}
In this subsection we prove Lem. \ref{lem:adaptive} and \ref{lem:unique} that lead to Thm.~\ref{thm:opt_one_step}. 

\subsubsection*{Proof of Lem. \ref{lem:adaptive}}
Since any single-bit message $M(u) \in \{-1,1\}$ is characterized by two decision region $A_1 = M^{-1}(1)$ and $A_{-1} = M^{-1}(-1)$, it follows that $\mathbb E \left[ U | M(U) \right]$ assumes only two values: $\mu_1 = \mathbb E \left[ U | M(U) = 1 \right]$ and $\mu_{-1} = \mathbb E \left[ U | M(U) = -1 \right]$. We claim that a necessary condition for $M(u)$ to be optimal is that the sets $A_1$ and $A_{-1}$ are, modulo a set of measure $P(du)$ zero, the Voronoi sets on $\mathbb R$ corresponding to the points $\mu_1$ and $\mu_{-1}$, respectively. Indeed, assume by contradiction that for such an optimal partition there exists a set $B \subset A_{1}$ with $\mathbb P (U \in B) >0$ such that $\left( b-\mu_{1} \right)^2 > \left( b- \mu_{-1} \right)^2$. The expected square error in this partition satisfies:
\begin{align*}
& \int_{\mathbb R} \left( u - \mathbb E[U|M(u)]  \right)^2 P(du) \\
& = \int_{A_1} (u- \mu_1)^2 P(du) + \int_{A_{-1}} (u- \mu_{-1})^2 P(du) \\
& = \int_{A_1\setminus B} (u- \mu_1)^2 P(du) +  \int_{B} (u- \mu_1)^2 P(du) \\
&\quad \quad + \int_{A_{-1}} (u- \mu_{-1})^2 P(du) \\
& > \int_{A_1\setminus B} (u- \mu_1)^2 P(du) +  \int_{B} (u- \mu_2)^2 P(du) \\
& + \quad \quad  \int_{A_{-1}} (u- \mu_{-1})^2 P(du),
\end{align*}
so clearly, the partition $A_1' = A_1 \setminus B$, $A_{-1}' = A_{-1} \cup B$ attains lower error variance, what contradicts the optimality assumption and proves our claim. It is evident that Voronoi partition of the real line corresponding to $\mu_1$ and $\mu_{-1}$ is of the form $A_{-1} = (-\infty,\tau)$, $A_{1} = (\tau, \infty)$ where the point $\tau$ is of equal distance from $\mu_1$ and $\mu_{-1}$, namely $\tau = \frac{\mu_1 + \mu_{-1}}{2}$. From these two conditions (which are a special case of the conditions derived in \cite{1056489} for two quantization regions) we conclude that $\tau$ must satisfy the equation
\[
2 \tau = \frac{\int_{\tau}^\infty u P(du)}{\int_{\tau}^\infty P(du)} + \frac{\int_{-\infty}^{\tau} u P(du)}{\int_{-\infty}^{\tau} P(du)}.
\] 
\QEDA

\subsubsection*{Proof of Lem. \ref{lem:unique}} 
Any solution to \eqref{eq:lem_fixed_point} is a solution to $h^+(x) = h^-(x)$ where
\[
h^+(x) = \frac{\int_x^\infty uf(u)du}{\int_x^\infty f(u)du} - x 
\]
and
\[
h^-(x) = x - \frac{\int_{-\infty}^x uf(u)du}{\int_{-\infty}^x f(u)du}.
\]
We now prove that $h^+(x)$ is monotonically decreasing while $h^-(x)$ is increasing, so they meet at most at one point. The derivative of $h^-(x)$ is given by
\begin{equation} 
\label{eq:one_step_proof_derivative}
1 - \frac{ f(\tau) \int_{-\infty}^\tau f(x) (\tau-x)dx } {\left( \int_{-\infty}^\tau f(x) dx \right)^2}.
\end{equation}
Denote $F(x) = \int_{-\infty}^x f(u)du$. Using integration by parts in the numerator and from the fact that $\lim_{\tau \rightarrow -\infty}  \tau \int_{-\infty}^\tau f(x) dx = 0$, the last expression can be written as
\[
1- \frac{ f(\tau) \int_{-\infty}^\tau F(x) dx}
 {\left( F(\tau) \right)^2}.
\]
Log-concavity of $f(x)$ implies log-concavity of $F(x)$, so that we can write $F(x) = e^{g(x)}$ for some concave and differentiable function $g(x)$. Moreover, we have $f(x) = g'(x)e^{g(x)}$ where, by concavity of $g(x)$, the derivative $g'(x)$ of $g(x)$ is non-increasing. With these notation we have
\begin{align*}
\frac{ f(\tau) \int_{-\infty}^\tau F(x) dx}
 {\left( F(\tau) \right)^2} & = \frac{g'(\tau)e^{g(\tau)} \int_{-\infty}^\tau e^{g(x)}dx }{ e^{2g(\tau)} } \\
 & = e^{-g(\tau)} \int_{-\infty}^\tau g'(\tau) e^{g(x)} dx  \\
 & \leq e^{-g(\tau)} \int_{-\infty}^\tau g'(x) e^{g(x)} dx \\
 & = e^{-g(\tau)} F(\tau) = 1.
\end{align*}
(where the second from the last step follows since $g'(x) \leq g'(\tau)$ for any $x\leq \tau$). If follows that \eqref{eq:one_step_proof_derivative} is non-negative and thus $h^-(x)$ is monotonically increasing. Since
\[
h^+(-x) = x - \frac{ \int_{-\infty}^{x} uf(-u)du}{ \int_{-\infty}^x f(-u) du }, 
\]
the fact that $h^+(x)$ is monotonically decreasing follows from similar arguments. Moreover, since the derivatives of $h^+(x)$ and $h^-(x)$ never vanish at the same time over any open interval, their difference cannot be constant over any interval. Finally, since 
\begin{align*}
 \lim_{x\rightarrow -\infty} h^+(x) =  \lim_{x\rightarrow \infty} h^-(x)
\end{align*}
and since non of these functions are constant, monotonicity of $h^+(x)$ and $h^-(x)$ implies that they must meet at a single point in $\mathbb R$. \QEDA


\bibliographystyle{IEEEtran}
\bibliography{IEEEabrv,/Users/Alon1/LaTex/bibtex/sampling}

\end{document}